\documentclass[11pt]{amsart}
\usepackage{amsfonts}
\allowdisplaybreaks[4]
\newtheorem{theorem}{Theorem}

\newtheorem{corollary}{Corollary}

\newtheorem{lemma}{Lemma}

\newtheorem{remark}{Remark}

\numberwithin{equation}{section}

\newtheorem{definitionalph}{Definition}

\begin{document}
\title[Ostrowski type fractional integral inequalities for $h$-convex functions]{Some Ostrowski type inequalities via Riemann-Liouville fractional integrals for $h-$convex functions}
\author{Wenjun Liu}
\address{College of Mathematics and Statistics, Nanjing University of
Information Science and Technology, Nanjing 210044, China}
\email{wjliu.cn@gmail.com}

\begin{abstract}
In this paper, some Ostrowski type  inequalities via Riemann-Liouville fractional integrals for $h-$convex functions, which are super-multiplicative or
super-additive,  are given. These results not only generalize those of \cite{s2012,t2012}, but also provide new estimates on these types of Ostrowski
inequalities for fractional integrals.
\end{abstract}

\maketitle

\section{Introduction}

Let $f:I\rightarrow \mathbb{R},$ where $I\subseteq\mathbb{R}$ is an interval, be a mapping differentiable in the interior $I^{\circ }$
of $I$, and let $a,b\in I^{\circ }$ with $a<b$. If $\left| f^{\prime
}\left( x\right) \right| \leq M$ for all $x\in \left[ a,b\right] $, then
\begin{equation}
\left| f(x)-\frac{1}{b-a}\int_{a}^{b}f(t)dt\right| \leq M\left(
b-a\right) \left[ \frac{1}{4}+\frac{\left( x-\frac{a+b}{2}\right) ^{2}}{%
\left( b-a\right) ^{2}}\right], \quad \forall\ x\in \left[ a,b\right].   \label{e1}
\end{equation}%
  This  is the well-known  Ostrowski inequality (see \cite{OST} or \cite[page 468]{Mitrinovic1}), which gives an upper bound for the approximation of the integral
average $\frac{1}{(b-a)}\int_{a}^{b}f(t)dt$ by the value $f(x)$ at point $%
x\in \left[ a,b\right].$  In recent years, a number of
authors have written about generalizations, extensions and variants of  such inequalities  (see \cite{ADDC,D1,D3,ln,Liu,Zhongxue,sarikaya1}).  

Let us recall definitions of some  kinds of convexity as follows.

\begin{definitionalph}
  \cite{god} We say that $f:I\rightarrow\mathbb{R}$  is a Godunova-Levin function or that  $f$ belongs to the
class  $Q(I) $  if  $f$ is non-negative and for
all  $x,y\in I$  and  $t\in (0,1)$, one has
\begin{equation*}
f(tx+(1-t) y) \leq \frac{f(x) }{t}+\frac{%
f(y) }{1-t}.
\end{equation*}
\end{definitionalph}

\begin{definitionalph}
\cite{dr1} We say that  $f:I\subseteq\mathbb{R}\rightarrow\mathbb{R}$ is a  $P-$function or that $f$ belongs to the
class $P(I) $ if  $f$ is non-negative and for
all $x,y\in I$ and  $t\in [0,1],$ one has
\begin{equation*}
f(tx+(1-t) y) \leq f(x) +f(
y).
\end{equation*}
\end{definitionalph}

\begin{definitionalph}
\cite{hud}  We say that $f:(0,\infty ] \rightarrow [0,\infty ] $  is
 $s-$convex in the second sense, or that $f$ belongs to
the class  $K_{s}^{2}$, if  for all $x,y\in (0,b] $, $t\in [0,1%
] $ and for some fixed $s\in (0,1]$, one has
\begin{equation*}
f(tx+(1-t) y) \leq t^{s}f(x) +(
1-t) ^{s}f(y).
\end{equation*}%
\end{definitionalph}

\begin{definitionalph}
\cite{var}  Let  $h:J\subseteq\mathbb{R}\rightarrow\mathbb{R}$  be a positive function. We say that $f:I\subseteq\mathbb{R}\rightarrow\mathbb{R}$  is $h-$convex, or that $f$ belongs to
the class $SX(h,I) $,  if  $f$ is non-negative
and for all $x,y\in I$  and  $t\in [0,1]$, one has
\begin{equation}
f(tx+(1-t) y) \leq h(t) f(x)
+h(1-t) f(y).  \label{106}
\end{equation}
If inequality (\ref{106}) is reversed, then $f$ is said to be $h-$%
concave, i.e. $f\in SV(h,I) $.
\end{definitionalph}

 If $h(t)=t $, then all non-negative convex functions belong to $SX(h,I) $ and all non-negative concave functions belong to $SV(h,I) $; if
$h(t) =\frac{1}{t}$, then $SX(h,I) =Q(I)$; if $h(t) =1$, then $SX(h,I) \supseteq P(I) $; and if $h(t) =t^{s}$ for $s\in (0,1] $, then $SX(h,I) \supseteq K_{s}^{2}$.

\begin{remark}
   \cite{var} Let $h$ be a non-negative function such that $h(t ) \geq t $
for all $t \in (0,1)$.   If $f$ is a non-negative convex
function on $I$, then for $x,y\in I$, $t \in (0,1)$, one has
\begin{equation}
f(t x+(1-t )y) \leq t f(x)+(1-t )f(y)\leq
h(t )f(x)+h(1-t )f(y).  \label{108}
\end{equation}%
So, $f\in SX(h,I)$. Similarly, if the function $h$ has the property: $%
h(t )\leq t $ for all $t \in (0,1)$, then any non-negative
concave function $f$ belongs to the class $SV(h,I)$.
\end{remark}

\begin{definitionalph}
\cite{var} We say that $h:J\rightarrow\mathbb{R}$ is a super-multiplicative function, if for all $x,y\in J,$ one has
\begin{equation*}
h(xy) \geq h(x) h(y).
\end{equation*}
\end{definitionalph}

\begin{definitionalph}
\cite{alzer} We say that $h:J\rightarrow\mathbb{R}$ is   a super-additive function, if for all $x,y\in J,$ one has
\begin{equation*}
h(x+y) \geq h(x) +h(y).
\end{equation*}%
\end{definitionalph}

For recent results concerning $h-$convex functions see \cite{bom,zeki,t2012,var} and references therein. More recently, Tunc \cite{t2012} established some new Ostrowski type inequalities for the class of $h-$convex functions which are super-multiplicative or super-additive.

We then recall some definitions and mathematical preliminaries of
fractional calculus theory which will be used throughout this paper.

\begin{definitionalph}
Let $f\in L_{1}[a,b].$ The Riemann-Liouville integrals $J_{a+}^{\alpha }f$
and $J_{b-}^{\alpha }f$ of order $\alpha >0$ with $a\geq 0$ are defined by
\begin{equation*}
J_{a+}^{\alpha }f(x)=\frac{1}{\Gamma (\alpha )}\int_{a}^{x}\left( x-t\right)
^{\alpha -1}f(t)dt,\ \ x>a
\end{equation*}%
and%
\begin{equation*}
J_{b-}^{\alpha }f(x)=\frac{1}{\Gamma (\alpha )}\int_{x}^{b}\left( t-x\right)
^{\alpha -1}f(t)dt,\ \ x<b,
\end{equation*}
respectively, where $\Gamma (\alpha )=\int_{0}^{\infty }e^{-u}u^{\alpha -1}du$. Here, $%
J_{a+}^{0}f(x)=J_{b-}^{0}f(x)=f(x).$
\end{definitionalph}

In the case of $\alpha =1,$ the fractional integral reduces to the classical
integral. For some recent results connected with fractional integral inequalities we refer the reader to the papers \cite{Anastassiou,Belarbi,Dahmani1,d2011,Gorenflo,Miller,Podlubny,sarikaya2}  and the reference cited therein. In \cite{s2012}, Set established some new
Ostrowski type inequalities for $s-$convex functions in the second sense via Riemann-Liouville fractional integral.

Motivated by these results, in the present paper, we
establish some Ostrowski type inequalities via Riemann-Liouville fractional integrals for $h-$convex functions, which are super-multiplicative or
super-additive. So, new estimates on these types of Ostrowski inequalities via fractional integrals are provided and the results of \cite{s2012,t2012} are generalized.

\section{Ostrowski type fractional integral inequalities for $h$-convex functions}

To prove our main theorems, we need the following identity  established  by Set in  \cite{s2012}:

\begin{lemma}
\label{L1}  Let $f:\left[ a,b\right] \rightarrow \mathbb{R}$ be a
differentiable mapping on $(a,b)$ with $a<b.$ If $f'\in L_1\left[ a,b%
\right],$ then for all $x\in \left[ a,b\right] $ and $\alpha >0$, one has
\begin{align}  \label{E1}
&\left( \frac{\left( x-a\right) ^{\alpha }+\left( b-x\right) ^{\alpha }}{b-a%
}\right) f(x)-\frac{\Gamma (\alpha +1)}{\left( b-a\right) }\left[
J_{x-}^{\alpha }f(a)+J_{x+}^{\alpha }f(b)\right] \notag \\
=&\frac{\left( x-a\right) ^{\alpha +1}}{b-a}\int_{0}^{1}t^{\alpha
}f'\left( tx+(1-t)a\right) dt-\frac{\left( b-x\right) ^{\alpha +1}}{%
b-a}\int_{0}^{1}t^{\alpha }f'\left( tx+(1-t)b\right) dt.
\end{align}%
\end{lemma}

Using this lemma, we can obtain the following fractional integral inequalities for $h-$convex functions.

\begin{theorem}
\label{T1} Let $h:J\subseteq\mathbb{R}\rightarrow\mathbb{R}$  $([0,1]\subseteq J)$ be a non-negative and super-multiplicative function, $h\left( t \right) \geq t$ for $0\le t\le 1$, $f:\left[ a,b\right] \subset \lbrack 0,\infty )\rightarrow
\mathbb{R}$  be a differentiable mapping on $(a,b)$ with $a<b$ such that $%
f'\in L_1\left[ a,b\right].$ If $\left| f'\right| $
is $h-$convex  on $[a,b]$  and
$\left| f'(x)\right| \leq M,$ $x\in \left[ a,b\right],$
then the following inequalities for fractional integrals with $\alpha >0$
hold:%
\begin{align}
&\left| \left( \frac{\left( x-a\right) ^{\alpha }+\left( b-x\right)
^{\alpha }}{b-a}\right) f(x)-\frac{\Gamma (\alpha +1)}{\left( b-a\right) }%
\left[ J_{x-}^{\alpha }f(a)+J_{x+}^{\alpha }f(b)\right] \right|   \notag \\
\leq &\frac{M\left[ \left( x-a\right)
^{\alpha +1}+\left( b-x\right) ^{\alpha +1}\right]}{b-a}\int_{0}^{1}\left[t^{\alpha}h\left(t\right)   +t^{\alpha}h\left(1-t\right) \right] dt \label{2.2} \\
\leq &\frac{M\left[ \left( x-a\right)
^{\alpha +1}+\left( b-x\right) ^{\alpha +1}\right]}{b-a}\int_{0}^{1}\left[h\left(t^{\alpha+1}\right)   +h\left(t^{\alpha}(1-t)\right) \right] dt.  \label{E6}
\end{align}%
\end{theorem}

\begin{proof}
From (\ref{E1}) and since $\left| f'\right| $ is  $h-$%
convex, we have
\begin{align*}
&\left| \left( \frac{\left( x-a\right) ^{\alpha }+\left( b-x\right)
^{\alpha }}{b-a}\right) f(x)-\frac{\Gamma (\alpha +1)}{\left( b-a\right) }%
\left[ J_{x-}^{\alpha }f(a)+J_{x+}^{\alpha }f(b)\right] \right|  \\
\leq &\frac{\left( x-a\right) ^{\alpha +1}}{b-a}\int_{0}^{1}t^{\alpha
}\left| f'\left( tx+(1-t)a\right) \right| dt +\frac{\left( b-x\right) ^{\alpha +1}}{b-a}\int_{0}^{1}t^{\alpha
}\left| f'\left( tx+(1-t)b\right) \right| dt \\
\leq &\frac{\left( x-a\right) ^{\alpha +1}}{b-a}\int_{0}^{1}\left[t^{\alpha}h(t)\left| f'(x)\right| +t^{\alpha }h(1-t)\left|
f'(a)\right|\right] dt \\
&+\frac{\left( b-x\right) ^{\alpha +1}}{b-a}\int_{0}^{1}\left[t^{\alpha}h(t)\left| f'(x)\right| +t^{\alpha }h(1-t)\left|
f'(b)\right|\right] dt \\
\leq &\frac{M\left( x-a\right) ^{\alpha +1}}{b-a}\int_{0}^{1}\left[t^{\alpha}h(t) +t^{\alpha }h(1-t) \right] dt +\frac{M\left( b-x\right) ^{\alpha +1}}{b-a}\int_{0}^{1}\left[t^{\alpha}h(t)  +t^{\alpha }h(1-t) \right] dt,
\end{align*}
which completes the  proof of   \eqref{2.2}.

By using the additional
properties of $h$ in the assumptions, we further have
\begin{align}
 \int_{0}^{1}\left[t^{\alpha}h(t)  +t^{\alpha }h(1-t) \right] dt
\leq &\int_{0}^{1}\left[h\left(t^{\alpha}\right)h(t)  +h\left(t^{\alpha}\right)h(1-t)\right] dt  \notag \\
\leq & \int_{0}^{1}\left[h\left(t^{\alpha+1}\right)  +h\left(t^{\alpha}(1-t)\right)\right] dt. \label{2.3}
\end{align}%
Hence,    the  proof of   \eqref{E6} is complete.
\end{proof}

\begin{remark}
\label{R1}
We note that in the  proof of \eqref{2.2} we does not use the additional super-multiplicative property of $h$ and the condition ``$h\left( t \right) \geq t$ for $0\le t\le 1$".
In Theorem \ref{T1}, if we choose $\alpha =1$, then  \eqref{E6}
reduces the inequality \cite[(2.1)]{t2012}, i.e.,
\begin{align*}
\left\vert f\left( x\right) -\frac{1}{b-a}\int_{a}^{b}f\left( u\right)
du\right\vert \leq \frac{M\left[ \left( x-a\right) ^{2}+\left( b-x\right)
^{2}\right] }{b-a}\int_{0}^{1}\left[ h\left( t^{2}\right) +h\left(
t-t^{2}\right) \right] dt,
\end{align*}
which can be better than the inequality  \eqref{e1} provide that   $h$  is chosen such that
$$\int_{0}^{1}\left[ h\left( t^2\right) +h\left( t-t^2\right) \right]
dt<\frac{1}{2}.$$  In Theorem \ref{T1}, if we choose $h(t)=t$, then  \eqref{2.2} and \eqref{E6}
reduce  the inequality in  \cite[Corollary 1]{s2012}.
\end{remark}

In the next corollary, we will also make use of the Beta function of Euler
type, which is  defined as
\begin{equation*}
\beta \left( x,y\right) =\int_{0}^{1}t^{x-1}\left( 1-t\right) ^{y-1}dt=\frac{%
\Gamma \left( x\right) \Gamma \left( y\right) }{\Gamma \left( x+y\right) }, \quad \forall\ x,y>0.
\end{equation*}

\begin{corollary}
If we choose $h\left( t\right) =t^{s}$,   $s\in \left( 0,1\right] $, in Theorem \ref{T1}, then
we have%
\begin{align*}
&\left|\left( \frac{\left( x-a\right) ^{\alpha }+\left( b-x\right) ^{\alpha }}{b-a%
}\right) f(x)-\frac{\Gamma (\alpha +1)}{\left( b-a\right) }\left[
J_{x-}^{\alpha }f(a)+J_{x+}^{\alpha }f(b)\right] \right|\notag \\
\leq &\frac{M}{b-a}\left[1+\frac{\Gamma \left(\alpha+1\right) \Gamma \left(s+1\right) }{\Gamma \left( \alpha +s+1\right) }\right]\frac{  \left( x-a\right)
^{\alpha +1}+\left( b-x\right) ^{\alpha +1} }{ \alpha +s+1 }\notag \\
\leq &\frac{M}{b-a}\left[1+\frac{\Gamma \left(\alpha s+1\right) \Gamma \left(s+1\right) }{\Gamma \left( \alpha s+s+1\right) }\right]\frac{  \left( x-a\right)
^{\alpha +1}+\left( b-x\right) ^{\alpha +1} }{ \alpha s+s+1 },
\end{align*}
due to the fact that
\begin{align*}
&\int_{0}^{1}\left[h\left(t^{\alpha+1}\right)   +h\left(t^{\alpha}(1-t)\right) \right] dt
=\int_{0}^{1}  t^{s(\alpha+1)} dt+ \int_{0}^{1} t^{\alpha s}(1-t)^s    dt\\
=& \frac{1}{\alpha s+s+1}+\frac{\Gamma \left(\alpha s+1\right) \Gamma \left(s+1\right) }{\Gamma \left( \alpha s+s+2\right) }
= \frac{1}{\alpha s+s+1}\left[1+\frac{\Gamma \left(\alpha s+1\right) \Gamma \left(s+1\right) }{\Gamma \left( \alpha s+s+1\right) }\right].
\end{align*}
The first inequality is the same as the one established in \cite[Theorem 7]{s2012}.
 \end{corollary}

\begin{theorem}
\label{T2} Let $h:J\subseteq\mathbb{R}\rightarrow\mathbb{R}$  $([0,1]\subseteq J)$ be a non-negative and super-additive function, and $f:\left[ a,b\right] \subset \lbrack 0,\infty )\rightarrow
\mathbb{R}$  be a differentiable mapping on $(a,b)$ with $a<b$ such that $%
f'\in L_1\left[ a,b\right].$ If $\left| f'\right|^q $
is $h-$convex  on $[a,b]$, $p,q>1$, $\frac{1}{p}+\frac{1}{q}=1,$  and
$\left| f'(x)\right| \leq M,$ $x\in \left[ a,b\right],$
then the following inequality for fractional integrals with $\alpha >0$
holds:%
\begin{align}
&\left| \left( \frac{\left( x-a\right) ^{\alpha }+\left( b-x\right)
^{\alpha }}{b-a}\right) f(x)-\frac{\Gamma (\alpha +1)}{\left( b-a\right) }%
\left[ J_{x-}^{\alpha }f(a)+J_{x+}^{\alpha }f(b)\right] \right| \notag \\
\leq &\frac{M\left[ \left( x-a\right)
^{\alpha +1}+\left( b-x\right) ^{\alpha +1}\right]}{\left( 1+p\alpha \right) ^{\frac{1}{p}}(b-a)} \left(\int_{0}^{1}\left[ h\left( t\right) +h\left( 1-t\right) \right]
dt\right)^{\frac{1}{q}} \label{E7}\\
\leq &\frac{M\left[ \left( x-a\right)
^{\alpha +1}+\left( b-x\right) ^{\alpha +1}\right]}{\left( 1+p\alpha \right) ^{\frac{1}{p}}(b-a)}  h^{\frac{1}{q}}(1).   \label{2.6}
\end{align}%
\end{theorem}

\begin{proof}
From Lemma \ref{L1} and using the well-known H\"{o}lder's inequality, we have%
\begin{align*}
&\left| \left( \frac{\left( x-a\right) ^{\alpha }+\left( b-x\right)
^{\alpha }}{b-a}\right) f(x)-\frac{\Gamma (\alpha +1)}{\left( b-a\right) }%
\left[ J_{x-}^{\alpha }f(a)+J_{x+}^{\alpha }f(b)\right] \right|  \\
\leq &\frac{\left( x-a\right) ^{\alpha +1}}{b-a}\int_{0}^{1}t^{\alpha
}\left| f'\left( tx+(1-t)a\right) \right| dt  +\frac{\left( b-x\right) ^{\alpha +1}}{b-a}\int_{0}^{1}t^{\alpha
}\left| f'\left( tx+(1-t)b\right) \right| dt \\
\leq &\frac{\left( x-a\right) ^{\alpha +1}}{b-a}\left(
\int_{0}^{1}t^{p\alpha }dt\right) ^{\frac{1}{p}}\left(
\int_{0}^{1}\left| f'\left( tx+(1-t)a\right) \right|
^{q}dt\right) ^{\frac{1}{q}} \\
&+\frac{\left( b-x\right) ^{\alpha +1}}{b-a}\left( \int_{0}^{1}t^{p\alpha
}dt\right) ^{\frac{1}{p}}\left( \int_{0}^{1}\left| f'\left(
tx+(1-t)b\right) \right| ^{q}dt\right) ^{\frac{1}{q}}.
\end{align*}%
Since $\left| f'\right| ^{q}$ is $h-$convex and $\left| f'(x)\right| \leq M$, we get
\begin{align*}
\int_{0}^{1}\left\vert f^{\prime }\left( tx+\left( 1-t\right) a\right)
\right\vert ^{q}dt \leq &\int_{0}^{1}\left[ h\left( t\right) \left\vert
f^{\prime }\left( x\right) \right\vert ^{q}+h\left( 1-t\right) \left\vert
f^{\prime }\left( a\right) \right\vert ^{q}\right] dt \\
\leq &M^{q}\int_{0}^{1}\left[ h\left( t\right) +h\left( 1-t\right) \right]
dt
\end{align*}
and similarly
\begin{align*}
\int_{0}^{1}\left\vert f^{\prime }\left( tx+\left( 1-t\right) b\right)
\right\vert ^{q}dt
\leq    M^{q}\int_{0}^{1}\left[ h\left( t\right) +h\left( 1-t\right) \right]
dt.
\end{align*}
By simple computation, we have
\begin{equation*}
\int_{0}^{1}t^{p\alpha }dt=\frac{1}{p\alpha +1}.
\end{equation*}
Using these results, we complete the  proof of   \eqref{E7}.

 By using the super-additive property of $h$ in the assumptions, we further have
\begin{align*}
 \int_{0}^{1}\left[ h\left( t\right) +h\left( 1-t\right) \right]
dt\leq \int_{0}^{1} h(1) dt=h(1).
\end{align*}%
Hence,    the  proof of   \eqref{2.6} is complete.
\end{proof}

\begin{remark}
\label{R2}
We note that in the  proof of \eqref{E7} we does not use the additional super-additive property of $h$.
In Theorem \ref{T2},  if we choose $h(t)=t$, then  \eqref{E7}
reduces the inequality in  \cite[Corollary 2]{s2012}; in Theorem \ref{T2}, if we choose $\alpha =1$, then  \eqref{E7}
becomes
\begin{align}
 \left| f(x)-\frac{1}{b-a}\int_{a}^{b}f(t)dt\right|
\leq  \frac{M\left[ \left( x-a\right)
^{2}+\left( b-x\right) ^{2}\right]}{\left( 1+p \right) ^{\frac{1}{p}}(b-a)}\left(\int_{0}^{1}\left[ h\left( t\right) +h\left( 1-t\right) \right]
dt\right)^{\frac{1}{q}}, \label{E7'}
\end{align}
which can be better than the inequality  \eqref{e1} provide that $p, q$ and $h$  are chosen such that
$$\left(\int_{0}^{1}\left[ h\left( t\right) +h\left( 1-t\right) \right]
dt\right)^{\frac{1}{q}}<\frac{1}{2}\left( 1+p \right)^{\frac{1}{p}}.$$
\end{remark}

\begin{corollary}
If we choose $h\left( t\right) =t^{s}$,   $s\in \left( 0,1\right]$, in Theorem \ref{T2}, then
we have%
\begin{align*}
&\left|\left( \frac{\left( x-a\right) ^{\alpha }+\left( b-x\right) ^{\alpha }}{b-a%
}\right) f(x)-\frac{\Gamma (\alpha +1)}{\left( b-a\right) }\left[
J_{x-}^{\alpha }f(a)+J_{x+}^{\alpha }f(b)\right] \right|\notag \\
\leq &\frac{M}{\left( 1+p\alpha \right) ^{\frac{1}{p}}}\left( \frac{2}{s+1}%
\right) ^{\frac{1}{q}}\frac{\left( x-a\right) ^{\alpha +1}+\left(
b-x\right) ^{\alpha +1}}{b-a},
\end{align*}
due to the fact that
\begin{align*}
 \int_{0}^{1}\left[ h\left( t\right) +h\left( 1-t\right) \right]
dt
= \frac{2}{s+1}.
\end{align*}
This is the inequality established in \cite[Theorem 8]{s2012}.
 \end{corollary}

\begin{theorem}
\label{T3} Let $h:J\subseteq\mathbb{R}\rightarrow\mathbb{R}$  $([0,1]\subseteq J)$ be a non-negative and super-multiplicative function, $h\left( t \right) \geq t$ for $0\le t\le 1$, $f:\left[ a,b\right] \subset \lbrack 0,\infty )\rightarrow
\mathbb{R}$  be a differentiable mapping on $(a,b)$ with $a<b$ such that $%
f'\in L_1\left[ a,b\right].$ If $\left| f'\right|^q $
is $h-$convex  on $[a,b]$, $q\ge 1$  and
$\left| f'(x)\right| \leq M,$ $x\in \left[ a,b\right],$
then the following inequalities for fractional integrals with $\alpha >0$
hold:%
\begin{align}
&\left| \left( \frac{\left( x-a\right) ^{\alpha }+\left( b-x\right)
^{\alpha }}{b-a}\right) f(x)-\frac{\Gamma (\alpha +1)}{\left( b-a\right) }%
\left[ J_{x-}^{\alpha }f(a)+J_{x+}^{\alpha }f(b)\right] \right|  \notag   \\
\leq &\frac{M}{\left( 1+\alpha \right) ^{1-\frac{1}{q}}}  \frac{\left( x-a\right) ^{\alpha +1}+\left(
b-x\right) ^{\alpha +1}}{b-a}\left( \int_{0}^{1}\left[t^{\alpha}h\left(t\right)   +t^{\alpha}h\left(1-t\right) \right] dt
\right) ^{\frac{1}{q}} \label{2.7}\\
\leq &\frac{M}{\left( 1+\alpha \right) ^{1-\frac{1}{q}}}  \frac{\left( x-a\right) ^{\alpha +1}+\left(
b-x\right) ^{\alpha +1}}{b-a}\left(\int_{0}^{1}\left[h\left(t^{\alpha+1}\right)   +h\left(t^{\alpha}(1-t)\right) \right] dt
\right) ^{\frac{1}{q}}. \label{E8}
\end{align}
\end{theorem}

\begin{proof}
From Lemma \ref{L1} and using the well-known power mean inequality, we have%
\begin{align*}
&\left| \left( \frac{\left( x-a\right) ^{\alpha }+\left( b-x\right)
^{\alpha }}{b-a}\right) f(x)-\frac{\Gamma (\alpha +1)}{\left( b-a\right) }%
\left[ J_{x-}^{\alpha }f(a)+J_{x+}^{\alpha }f(b)\right] \right|  \\
\leq &\frac{\left( x-a\right) ^{\alpha +1}}{b-a}\int_{0}^{1}t^{\alpha
}\left| f'\left( tx+(1-t)a\right) \right| dt +\frac{\left( b-x\right) ^{\alpha +1}}{b-a}\int_{0}^{1}t^{\alpha
}\left| f'\left( tx+(1-t)b\right) \right| dt \\
\leq &\frac{\left( x-a\right) ^{\alpha +1}}{b-a}\left(
\int_{0}^{1}t^{\alpha }dt\right) ^{1-\frac{1}{q}}\left(
\int_{0}^{1}t^{\alpha }\left| f'\left( tx+(1-t)a\right)
\right| ^{q}dt\right) ^{\frac{1}{q}} \\
&+\frac{\left( b-x\right) ^{\alpha +1}}{b-a}\left( \int_{0}^{1}t^{\alpha
}dt\right) ^{1-\frac{1}{q}}\left( \int_{0}^{1}t^{\alpha }\left|
f'\left( tx+(1-t)b\right) \right| ^{q}dt\right) ^{\frac{1}{q}}.
\end{align*}%
Since $\left| f'\right| ^{q}$ is $h-$convex  on $[a,b]$ and $\left| f'(x)\right| \leq M$, we get
\begin{align*}
\int_{0}^{1}t^{\alpha }\left| f'\left( tx+(1-t)a\right)
\right| ^{q}dt \leq &\int_{0}^{1}\left[ t^{\alpha }h(t)\left|
f'\left( x\right) \right| ^{q}+t^{\alpha }h(1-t)\left|
f'\left( a\right) \right| ^{q}\right] dt \\
\le &  M^{q} \int_{0}^{1}\left[ t^{\alpha }h(t) +t^{\alpha }h(1-t) \right] dt
\end{align*}%
and similarly
\begin{align*}
\int_{0}^{1}t^{\alpha }\left| f'\left( tx+(1-t)b\right)
\right| ^{q}dt  \leq  M^{q} \int_{0}^{1}\left[ t^{\alpha }h(t) +t^{\alpha }h(1-t) \right] dt.
\end{align*}
Using these inequalities, we complete the  proof of   \eqref{2.7}.

By using the additional
properties of $h$ in the assumptions, we further have \eqref{2.3}.
Hence,    the  proof of   \eqref{E8} is complete.
\end{proof}

\begin{remark}
\label{R3}
We note that in the  proof of \eqref{2.7} we does not use the additional super-multiplicative property of $h$ and the condition ``$h\left( t \right) \geq t$ for $0\le t\le 1$".
In Theorem \ref{T3}, if we choose $\alpha =1$, then  \eqref{E8}
reduces the inequality \cite[(2.4)]{t2012};  in Theorem \ref{T3}, if we choose $h(t)=t$, then  \eqref{2.7} and \eqref{E8}
reduce  the inequality in  \cite[Corollary 3]{s2012}.
\end{remark}

\begin{corollary}
If we choose $h\left( t\right) =t^{s}$,   $s\in \left( 0,1\right] $, in Theorem \ref{T3}, then
we have%
\begin{align*}
&\left|\left( \frac{\left( x-a\right) ^{\alpha }+\left( b-x\right) ^{\alpha }}{b-a%
}\right) f(x)-\frac{\Gamma (\alpha +1)}{\left( b-a\right) }\left[
J_{x-}^{\alpha }f(a)+J_{x+}^{\alpha }f(b)\right] \right|\notag \\
\leq &\frac{M}{\left( 1+\alpha \right) ^{1-\frac{1}{q}}}\frac{1}{(\alpha +s+1)^{\frac{1}{q}}}\left[1+\frac{\Gamma \left(\alpha+1\right) \Gamma \left(s+1\right) }{\Gamma \left( \alpha s+s+1\right) }\right]^{\frac{1}{q}}\frac{  \left( x-a\right)
^{\alpha +1}+\left( b-x\right) ^{\alpha +1} }{ b-a }\notag \\
\leq &\frac{M}{\left( 1+\alpha \right) ^{1-\frac{1}{q}}}\frac{1}{(\alpha +s+1)^{\frac{1}{q}}}\left[1+\frac{\Gamma \left(\alpha s+1\right) \Gamma \left(s+1\right) }{\Gamma \left( \alpha s+s+1\right) }\right]^{\frac{1}{q}}\frac{  \left( x-a\right)
^{\alpha +1}+\left( b-x\right) ^{\alpha +1} }{ b-a }.
\end{align*}
The first inequality is the same as the one established in \cite[Theorem 9]{s2012}.
 \end{corollary}


\begin{thebibliography}{99}

\bibitem{ADDC} M. Alomari\ et al., Ostrowski type inequalities for functions whose derivatives are $s$-convex in the second sense, Appl. Math. Lett. {\bf 23} (2010), no.~9, 1071--1076.



\bibitem{alzer} \label{alzer} H. Alzer, A superadditive property of Hadamard's gamma function, Abh. Math. Semin. Univ. Hambg. {\bf 79} (2009), no.~1, 11--23.

\bibitem{Anastassiou} G. Anastassiou\ et al., Montgomery identities for fractional integrals and related fractional inequalities, JIPAM. J. Inequal. Pure Appl. Math. {\bf 10} (2009), no.~4, Article 97, 6 pp.





\bibitem{Belarbi} S. Belarbi\ and\ Z. Dahmani, On some new fractional integral inequalities, JIPAM. J. Inequal. Pure Appl. Math. {\bf 10} (2009), no.~3, Article 86, 5 pp.

\bibitem{bom} \label{bom} M. Bombardelli\ and\ S. Varo\v sanec, Properties of $h$-convex functions related to the Hermite-Hadamard-Fej\'er inequalities, Comput. Math. Appl. {\bf 58} (2009), no.~9, 1869--1877.

\bibitem{Dahmani1} Z. Dahmani, New inequalities in fractional integrals, Int. J. Nonlinear Sci. {\bf 9} (2010), no.~4, 493--497.



\bibitem{D1} S. S. Dragomir, The Ostrowski integral inequality for mappings of bounded variation, Bull. Austral. Math. Soc. {\bf 60} (1999), no.~3, 495--508.

\bibitem{D3} S. S. Dragomir, The Ostrowski's integral inequality for Lipschitzian mappings and applications, Comput. Math. Appl. {\bf 38} (1999), no.~11-12, 33--37.

\bibitem{dr1}   S. S. Dragomir, J. Pe\v cari\'c\ and\ L. E. Persson, Some inequalities of Hadamard type, Soochow J. Math. {\bf 21} (1995), no.~3, 335--341.

\bibitem{d2011}
B. Dyda, Fractional Hardy inequality with a remainder term, Colloq. Math. {\bf 122} (2011), no.~1, 59--67.

\bibitem{god}  E. K. Godunova\ and\ V. I. Levin, Inequalities for functions of a broad class that contains convex, monotone and some other forms of functions, in {\it Numerical mathematics and mathematical physics (Russian)}, 138--142, 166, Moskov. Gos. Ped. Inst., Moscow.

    \bibitem{Gorenflo} R. Gorenflo\ and\ F. Mainardi, Fractional calculus: integral and differential equations of fractional order, in {\it Fractals and fractional calculus in continuum mechanics (Udine, 1996)}, 223--276, CISM Courses and Lectures, 378 Springer, Vienna.

\bibitem{hud} \label{hud} H. Hudzik\ and\ L. Maligranda, Some remarks on $s$-convex functions, Aequationes Math. {\bf 48} (1994), no.~1, 100--111.

\bibitem{ln}
W. J. Liu\ and\ Q. -A. Ng\^ o, A generalization of Ostrowski inequality on time scales for $k$ points, Appl. Math. Comput. {\bf 203} (2008), no.~2, 754--760.

\bibitem{Liu} Z. Liu, Some companions of an Ostrowski type inequality and applications, JIPAM. J. Inequal. Pure Appl. Math. {\bf 10} (2009), no.~2, Article 52, 12 pp.

\bibitem{Zhongxue} Z. X. L\"u, On sharp inequalities of Simpson type and Ostrowski type in two independent variables, Comput. Math. Appl. {\bf 56} (2008), no.~8, 2043--2047.

\bibitem{Miller} K. S. Miller\ and\ B. Ross, {\it An introduction to the fractional calculus and fractional differential equations}, A Wiley-Interscience Publication, Wiley, New York, 1993.

\bibitem{Mitrinovic1} D. S. Mitrinovi\'c, J. E. Pe\v cari\'c\ and\ A. M. Fink, {\it Inequalities involving functions and their integrals and derivatives}, Mathematics and its Applications (East European Series), 53, Kluwer Acad. Publ., Dordrecht, 1991.

\bibitem{OST} A. Ostrowski, \"Uber die Absolutabweichung einer differentiierbaren Funktion von ihrem Integralmittelwert, Comment. Math. Helv. {\bf 10} (1937), no.~1, 226--227.

\bibitem{Podlubny} I. Podlubny, {\it Fractional differential equations}, Mathematics in Science and Engineering, 198, Academic Press, San Diego, CA, 1999.



\bibitem{sarikaya1} M. Z. Sarikaya, On the Ostrowski type integral inequality, Acta Math. Univ. Comenian. (N.S.) {\bf 79} (2010), no.~1, 129--134.

\bibitem{sarikaya2} M. Z. Sarikaya and H. Ogunmez, On new inequalities via Riemann-Liouville fractional integration, Abstr. Appl. Anal. 2012, Art. ID 428983, in press.

\bibitem{zeki} M. Z. Sarikaya, A. Saglam\ and\ H. Yildirim, On some Hadamard-type inequalities for $h$-convex functions, J. Math. Inequal. {\bf 2} (2008), no.~3, 335--341.


\bibitem{s2012}
E. Set, New inequalities of Ostrowski type for mappings whose derivatives are $s$-convex in the second sense via fractional integrals, Comput. Math. Appl. {\bf 63} (2012), no.~7, 1147--1154.


\bibitem{t2012}
M. Tunc, Ostrowski type inequalities via $h$-convex functions with applications for special means and P.D.F.'s, arXiv:1204.2921 [math.FA].


\bibitem{var}  S. Varo\v sanec, On $h$-convexity, J. Math. Anal. Appl. {\bf 326} (2007), no.~1, 303--311.




















\end{thebibliography}
\end{document}